\documentstyle[12pt,amsfonts,amssymb]{article}
\begin{document}
\newtheorem{theorem}{Theorem}[section]
\newtheorem{lemma}[theorem]{Lemma}
\newtheorem{defin}[theorem]{Definition}
\newtheorem{cor}[theorem]{Corollary}
\newtheorem{prop}[theorem]{Proposition}
\newtheorem{claim}[theorem]{Claim}
\newtheorem{remark}[theorem]{Remark}
\newenvironment{proof}{{\raggedright{\em Proof.}}}{\hspace{.1in}$\square$}
\newenvironment{poc}[1]{{\raggedright{\bf Proof of Claim {#1}.}}}{\hspace{
  .1in}$\square$}
\newenvironment{pol}[1]{{\raggedright{\bf Proof of Lemma {#1}.}}}{\hspace{
  .1in}$\square$}
\newenvironment{pop}[1]{{\raggedright{\bf Proof of Proposition {#1}.}}}
 {\hspace{ .1in}$\square$}
\newenvironment{pot}[1]{{\raggedright{\bf Proof of Theorem {#1}.}}}{\hspace{
  .1in}$\square$}
\newcounter{re}[theorem]
\newcounter{rom}
\newcommand{\rmn}{\stepcounter{rom}{\rm (\roman{rom})}}
\newcounter{gen}
\newcommand{\trm}[1]{\setcounter{gen}{#1}{\rm\roman{gen}}}
\newcommand{\ca}[1]{\mbox{$\Bbb{#1}$}}
\newcommand{\sca}[1]{\mbox{\scriptsize $\Bbb{#1}$}}
\newcommand{\tca}[1]{\mbox{\tiny $\Bbb{#1}$}}
\newcommand{\bol}[1]{\mbox{\boldmath ${#1}$}}
\newcommand{\Romn}[1]{\setcounter{gen}{#1}\Roman{gen}}

\title{An Approximation of Local Antiderivatives of
Relative Differential on Arithmetic
Surface}
\author{Yuhan Zha
\\{Institute of Mathematics}\\
{Chinese Academy of Sciences} \\{yhzha1@sina.com}} \maketitle

\begin{abstract}
\label{abstract} Let $\omega$ be a relative differential form on
an arithmetic surface $X$. We construct a family of rational
functions $G_x$ on  $X\otimes\ca{C}$, which can approximate local
antiderivatives of $\omega$ over an open set on $X\otimes\ca{C}$.
From this family of rational functions, we construct a rational
function $G_2$ on $X$. The function $G_2$ can generate an element
in the ring of integers of a number field, which can approximate
an inner product produced by $\omega$ and $\overline{\omega}$ over
an open set on $X_{\sca{C}}$. This will give a relation between
the height of a rational curve $E_P$ on $X$ and the norm of the
relative differential form $\omega$, which will give an upper
bound for the height of the rational curve $E_P$ under a few
assumptions.
\end{abstract}

\section{Introduction}
Let $R$ be the ring of integers of a number field $F$, and $Y={\rm
Spec}R$. Let $X$ be a stable family of curves of genus $g>1$ over
$Y$. Let $\pi:X\longrightarrow Y$ be the natural projection. Let
$E_P$ be a curve on $X$ that is rational over $Y$. Let
$\omega_{X/Y}$ be the canonical dualizing sheaf on $X$ endowed
with the canonical Hermitian metric [L1]. Let $\eta$ be the
canonical section of the line bundle ${\cal O}_X(E_P)$ vanishing
along $E_P$ on $X$. Let $S$ be the set of all the complex
embedding $\sigma:F\hookrightarrow\ca{C}$.

Let $\omega$ be a relative differential on $X$, such that there
exists sections $u_j$ of ${\cal O}_X(m_2E_P)$ on $X$ and sections
$e_j$ of ${\cal O}_X(m_3E_P)$ on $X$ satisfying
\begin{equation}\label{da8950} \sum_{j\in I_1}u_j\cdot e_j=0
\end{equation} and $
\omega=\sum_{j\in I_1}\frac{u_j}{\eta^{m_2}}\cdot
d\frac{e_j}{\eta^{m_3}}$ on $X$, where $m_2,m_3$ are positive
integers. Let $n_1$ be an integer satisfying $n_1>9g^2$. Let
$n_2=2n_1$. Let $\{\varrho_i:i\in I\}$ be the set of all the $n_1$
complex numbers satisfying $\varrho_i^{n_1}=1$. Let $n_2=2n_1$.
For $i=1,2$, let $\xi_i$ be a section of ${\cal O}_X(n_iE_P)$ on
$X$, that satisfy the following:
\begin{enumerate}\setcounter{rom}{0}\item[\rmn]
\begin{equation}
\label{da8843} \log\|\xi_i\|\leq
\log\|\xi_i\|_{h_1,\sigma}+O(1)\end{equation} for all the complex
embedding $\sigma\in S$ and $i=1,2$, where $\|\xi_i\|$ denotes the
norm of $\xi_i$ under the canonical Hermitian metric [L1] on
$R^0\pi_*(X,{\cal O}_X(n_iE_P))$, and $\|\xi_i\|_{h_1,\sigma}$
denotes the norm of $\xi_i$ under the canonical Hermitian metric
[L1] on the restriction of  ${\cal O}_X(n_iE_P)$ to
$E_P\otimes_\sigma\ca{C}$, and the constant implicit in $O(1)$ is
determined by $n_i$ and $X_{\sca{C}}$. \item[\rmn] For complex
embedding $\sigma\in S$, let $\{x_{j,\sigma}:j\in I_{2,\sigma}\}$
be the set of all the points on $X_\sigma$, where $\xi_1=0$. For
all $j\in I_{2,\sigma}$, there exists a simple connected open set
$U_{j,\sigma}$ on $X_\sigma$ containing $x_{j,\sigma}$, such that
\begin{equation}\label{da0600}
d\frac{\xi_1}{\eta^{n_1}}\neq 0
\end{equation} on $U_{j,\sigma}$, and $U_{i,\sigma}\cap
U_{j,\sigma}$ is empty for all $i,j\in I_{2,\sigma}$ satisfying
$i\neq j$, and
\begin{equation}\label{da0620}
\left|\frac{\xi_1}{\eta^{n_1}}\right|\geq a_1\cdot
\|\xi_1\|\end{equation} over $X_\sigma\smallsetminus \bigcup_{j\in
I_{2,\sigma}}U_{j,\sigma}$, where $a_1$ is a positive number
determined by $n_1,X_{\sca{C}}$. \item[\rmn] We have
$\|\xi_1\|>\frac{8}{a_1}$. For all $\sigma\in S$, we have the
following:

{\em For all $j\in I_{2,\sigma}$, there exists $i_j\in I$, such
that
\begin{equation}\label{da0610}
\left|\frac{\|\xi_2\|\cdot\eta^{n_2}}{\xi_2}(x)-\varrho_{i_j}\right|
<\frac{a_2}{\|\xi_1\|}\end{equation} for all $x\in U_{j,\sigma}$
satisfying $\left|\frac{\xi_1}{\eta^{n_1}}(x)\right|<2$, where
$a_2$ is a positive number determined by $n_1,X_{\sca{C}}$.  And
for all $j_1,j_2\in I_{2,\sigma}$ satisfying $j_1\neq j_2$, we
have $\varrho_{i_{j_1}}\neq\varrho_{i_{j_2}}$.} \item[\rmn] For
$i=1,2$, let $C_i$ be the divisor determined by $\xi_i=0$ as a
section of ${\cal O}_X(n_iE_P)$ on $E_P$. For convenience, we
denote the natural push forward cycle of $C_i$ on $Y$ still by
$C_i$. Then $\xi_i$ is a section of ${\cal O}_X(n_iE_P-C_i)$ on
$X$ for $i=1,2$.
\end{enumerate}

Under these assumptions, we construct a family of rational
functions $G_x$ on $X_{\sca{C}}$, that can approximate local
antiderivatives of $\omega$ over an open set on $X_{\sca{C}}$. And
from this family of functions, we construct a rational function
$G_2$ on $X$, that can generate an element in $R$, which can
approximate an inner product produced by $\omega$ and
$\overline{\omega}$ over an open set on $X_{\sca{C}}$ when the
canonical height $\omega_{X/Y}\cdot E_P$ is large enough. Note
when $\|\xi_i\|$ are large enough, we can pick $\xi_i$ such that
conditions (\trm{1}) (\trm{2}) (\trm{3}) are satisfied. When
$\omega_{X/Y}\cdot E_P$ is sufficiently large, $\|\xi_i\|$ will be
sufficiently large and it is possible to pick $\xi_i$ such that
(\trm{4}) is satisfied. So a theorem under assumptions above can
have applications. So for simplicity and to clarify what is new,
we will only prove our result under the assumptions
(\trm{1})-(\trm{4}) and assume $C_1=C_2=0$.

Since our construction works for number field case, and does not
work for function field case, so we first state a difference
between number fields and function fields. Then we use this
difference to construct the rational functions $G_x$ and $G_2$.

The difference happens over projective lines. Let $V$ be a two
dimensional vector space over $\ca{C}$ endowed with a Hermitian
metric. Let $\ca{P}_{\sca{C}}^1$ be the complex projective line
associated with $V$. Let ${\cal O}_{\sca{P}_{\tca{C}}^1}(1)$ be
the canonical line bundle of degree $1$ over $\ca{P}_{\sca{C}}^1$.
Let $h$ be the canonical Hermitian metric on ${\cal
O}_{\sca{P}_{\tca{C}}^1}(1)$ induced from the Hermitian metric on
$V$. Let $\{x_1,\cdots, x_n\}$ be a set of different points on
$\ca{P}_{\sca{C}}^1$, where $n>2$. Then for all $1\leq i\leq n$,
there exists an element $w_{x_i}\in V$, such that
\begin{equation}\label{da8800}\left|w_{x_i}(x_i)\right|_{h}=1
\end{equation} and
\begin{equation}\label{da8810}\left|w_{x_i}(x_j)\right|_{h}<1
\end{equation} for all $j\neq i$ satisfying $1\leq j\leq n$.
Note this fact is not true over function fields of characteristic
$p>0$. Over characteristic $p>0$, we can find a section of the
canonical line bundle of degree one over the projective line,
whose absolute value  is $<1$ at one point and is equal to $1$ at
the other $n-1$ points.

Now we use the fact above to construct the functions $G_x$
mentioned above. Let $m>0$ be a positive integer.  Let $V^{\otimes
m}$ be the vector space generated by $w_1\otimes\cdots \otimes
w_m$ for all $w_i\in V$. Let $S^m(V)$ be the submodule of
$V^{\otimes m}$ that is invariant under the action of the
symmetric group  on $m$ symbols, i.e. invariant under all the
permutations of $\{w_1,\cdots,w_m\}$. Let $M_m$ be the $m+1$
dimensional vector space generated by all the sections of ${\cal
O}_{\sca{P}_{\tca{C}}^1}(m)$ over $\ca{P}_{\sca{C}}^1$. Let's
consider the linear map $V^{\otimes 2m} \longrightarrow M_m
\otimes M_m$ defined by mapping $w_1\otimes\cdots\otimes w_{2m}$
to $\prod_{i=1}^mw_i \otimes \prod_{i=m+1}^{2m}w_i$, where $w_i\in
V$ for all $1\leq i\leq 2m$. By restricting the map above to the
submodule $S^{2m}(V)$ of $V^{\otimes 2m}$, we have the following
map
\begin{equation}\label{da8830} S^{2m}(V)\longrightarrow
M_m \otimes M_m
\end{equation} Since the natural map $S^{2m}(V)\longrightarrow
M_{2m}$ is an isomorphism, so (\ref{da8830}) determines the
following homomorphism
\begin{equation}\label{da8840}
f_5: M_{2m}\longrightarrow M_m\otimes M_m\end{equation}

Let $\{v_0,v_1\}$ be a set of basis of $V$ over $\ca{C}$. For
$i\in I$, let $x_i$ be the point on $\ca{P}_{\sca{C}}^1$, such
that
\begin{equation}\label{da8880}\frac{v_1}{v_0}(x_i)=\varrho_i
\end{equation}  For $i\in I$ and $r>0$,
let $U_{x_i}(r)$ be the open set determined by
$\left|\frac{v_1}{v_0}-\varrho_i\right|<r$ on
$\ca{P}_{\sca{C}}^1$. By the definition of $f_5$, we have
\begin{equation}\label{da8910}
f_5\left(v_0^mv_1^m\right)=\sum_{l=0}^{m}b_{m,l}\cdot
v_0^{l}v_1^{m-l} \otimes v_0^{m-l}v_1^{l}\end{equation} where
$b_{m,l}\in\ca{Q}$. For point $x\in \ca{P}_{\sca{C}}^1$, consider
the function
\begin{equation}\label{da8909}
f_{4,x}\left(v_0,v_1,m\right)=\sum_{l=0}^{m}b_{m,l}\cdot
\frac{v_1^{-l}}{v_0^{-l}}(x)\cdot \frac{v_1^{l}}{v_0^{l}}
\end{equation} over $\ca{P}_{\sca{C}}^1$.
 Then there exists $r_1,\rho_1\in (0,1)$ determined
by $\{\varrho_i:i\in I\}$, such that
\begin{equation}\label{da8908}
\left|f_{4,x}\left(v_0,v_1,m\right)(x')\right|<\rho_1^m
\end{equation} for all $m>0$ and
all $x\in U_{x_i}(r_1)$ and $x'\in U_{x_j}(r_1)$, where $i,j\in I$
satisfying $i\neq j$. And we have
\begin{equation}\label{da8907}
\left|f_{4,x}\left(v_0,v_1,{m}\right)(x')-1\right|\leq
2m\cdot\left|\frac{v_{1}}{v_0}(x')-\frac{v_{1}}{v_0}(x)\right|
\end{equation}
for all $i\in I$ and $x,x'\in U_{x_i}(r_1)$ satisfying
$\left|\frac{v_{1}}{v_0}(x')-\frac{v_{1}}{v_0}(x)\right|<\frac{1}{3m}$.

Let $z$ be an analytic function on $\bigcup_{i\in I}U_{x_i}(r_1)$,
such that $dz(x)\neq 0$ for all $x\in \bigcup_{i\in
I}U_{x_i}(r_1)$ satisfying $|z(x)|<r_0$, where $r_0$ is a positive
number. Let $U(r_0)$ be the open subset in $\bigcup_{i\in
I}U_{x_i}(r_1)$ determined by $|z(x)|<r_0$, for all $x\in
\bigcup_{i\in I}U_{x_i}(r_1)$. Let $\varphi_2$ be the map from
$U(r_0)$ to the complex plane defined by $z$. Let $U_0$ be the
open disc on the complex plane defined by $\left|z\right|<r_0$.
Assume there exists a unique point $x\in U_{x_i}(r_1)$, such that
$z(x)=\alpha$ for all $i\in I$ and $\alpha\in U_0$. Let
$\{u_j',e_j':j\in I_1\}$ be a set of analytic functions on
$U(r_0)$, such that
\begin{equation}\label{da8890}\sum_{j\in I_1}u_j'\cdot e_j'=0
\end{equation} Let
$\omega'= \sum_{j\in I_1}u_j'\cdot de_j'$. Assume $|u_j'|$,
$|e_j'|$, $\left|{\partial e_j'}/{\partial z}\right|$,
$\left|{\partial \frac{v_1}{v_0}}/{\partial z}\right|$ are less
than  $B'\cdot \left|\frac{\omega'}{dz}\right|$ over $U(r_0)$ for
all $j\in I_1$, where $B'>1$.

For $x\in U(r_0)$, let $G_x$ be the function on $U_0$ defined by
the following:
\begin{equation}\label{da8930}
G_x=\sum_{j\in I_1}\sum_{l=0}^{m}b_{m,l}\cdot
\frac{v_1^{-l}}{v_0^{-l}}(x)\cdot u_j'(x)\cdot \frac{z\cdot
\partial}{\partial z}{\rm Trace}_{\varphi_2}\left(\frac{
v_1^{l}}{v_0^{l}}\cdot e_j'\right)\end{equation}  Then by
(\ref{da8908}) (\ref{da8907}), there exists $\rho_2\in
(0,\rho_1)$, and $a_3,a_4>1$ which are determined by
$\rho_2,n_1,B',\omega'$, such that
\begin{equation}\label{da8940} \left|G_x(x')-\frac{\omega'}{dz}(x)
\cdot {z(x')}\right|\leq \left(a_4m\cdot \left|z(x')\right|^2
+n_1\cdot\rho_2^m\cdot
\left|z(x')\right|\right)\cdot\left|\frac{\omega'}{dz}(x)\right|
\end{equation} for all $x\in
U(r_0)$ and $m>a_3$ and  $x'\in U_0$ satisfying
\begin{equation}\label{da8931}
\left|z(x')-z(x)\right|<\frac{1}{3m}
\end{equation}
So $G_x$ can approximate an antiderivative of $\omega'$ locally
around point $x$, when $m$ is sufficiently large. Note this fact
is implied by the existence of $w_{x_i}$ discussed in the
paragraph (\ref{da8800})-(\ref{da8810}). Next we study the
algebraic and geometric implications of $G_x$ over  $X$.

Let $\omega$ be the relative differential on $X$ defined before.
Let
\begin{equation}\label{da9211}
s_\omega=\sum_{i\in I_1}\frac{u_i}{\eta^{m_2}}\otimes
\frac{e_i}{\eta^{m_3}}\end{equation} where $u_i,e_i$ are defined
before. Let $\xi_1,\xi_2$ be the sections of ${\cal
O}_{X}(n_1E_P),{\cal O}_{X}(n_2E_P)$ that satisfy the properties
(\trm{1})-(\trm{4}), where $C_1$ and $C_2$ are assumed to be zero.
Let $\ca{P}_Y^1$ be the projective line over $Y$.

Let
\begin{equation}\label{da8101} z=\frac{\xi_1}{
\|\xi_1\|\cdot\eta^{n_1}}\end{equation}
\begin{equation}\label{da6997}
\tau=\frac{\|\xi_2\| \cdot\eta^{n_2}}{ \xi_2}\end{equation} on
$X_{\sca{C}}$. Since $C_1=0$, so $\xi_1$ does not vanish at any
point on $E_P$, hence there exists a morphism
\begin{equation}\label{da0300}\varphi_2:X
\longrightarrow \ca{P}_Y^1
\end{equation} defined by
$[\xi_1:\eta^{n_1}]$. Let $\psi$ be the automorphism of
$\ca{P}_Y^1$ such that
\begin{equation}\label{da0310}\psi^*\left(\frac{\xi_1}{\eta^{n_1}}
\right)=\frac{\eta^{n_1}}{\xi_1}
\end{equation}
Let $G_2$ be the rational function on $X_{\sca{C}}$ defined by
\begin{equation}G_2(s_{\omega},
\tau,z)=\sum_{i\in I_1}\sum_{l=0}^{m}b_{m,l}\cdot \tau^{-l}\cdot
\frac{u_i}{\eta^{m_2}}\cdot
\varphi_2^*\left(\psi^*\left(\frac{z\cdot\partial}{\partial z}{\rm
Trace}_{\varphi_2}\left( \tau^{l} \cdot
\frac{e_i}{\eta^{m_3}}\right)\right)
\right)\label{da8130}\end{equation}

For complex embedding $\sigma:F\hookrightarrow\ca{C}$, let
$\Gamma_\sigma$ be the path determined by
$\left|\frac{\xi_1}{\eta^{n_1}}\right|=1$ on $X_\sigma$. Then
there exists an element in $R$ that is equal to the following:
\begin{eqnarray}\frac{(2m)!}{m!\cdot m!}\cdot
\frac{1}{2\pi\sqrt{-1}}\cdot \int_{\Gamma_{\sigma}} \omega\cdot
G_2 \label{da0502}\end{eqnarray} under the complex embedding
$\sigma$ for all $\sigma\in S$.

The reason is the following: Let $D_i$ be the divisor determined
by $\xi_i=0$ as a section of ${\cal O}_X(n_iE_P)$ on $X$ for
$i=1,2$. Let $D_3$ be the image of $D_2$ in $\ca{P}_Y^1$ under map
$\varphi_2$. Since $C_2=0$, so $D_2$ does not intersect with
$E_P$, therefore $D_3$ does not intersect with the image of $E_P$
in $\ca{P}_Y^1$ under map $\varphi_2$. So there exists a
polynomial
\begin{equation}\label{da0510}q=\sum_{j=0}^{n_2}
c_j\cdot\frac{\xi_1^j}{\eta^{jn_1}}
\end{equation} that vanishes along $D_3$, where $c_j\in
R$ and $c_{n_2}$ is a unit in $R$. So we have
\begin{equation}\label{da0520} {\rm Trace}_{\varphi_2}
\left(\frac{\eta^{n_2l}}{\xi_2^l}\cdot\frac{e_i}{\eta^{m_3}
}\right)=q^{-l}\cdot\left(\sum_{j=k_{i,l,1}}^{k_{i,l,2}}
c_{i,l,j}\cdot\frac{\xi_1^j}{\eta^{jn_1}}\right)
\end{equation} for all $0\leq l\leq m$ and $i\in I_1$,
where $k_{i,l,1},k_{i,l,2}$ are integers and $c_{i,l,j}\in R$. So
\begin{equation}\label{da0521}
\frac{z\cdot \partial}{\partial z}{\rm Trace}_{\varphi_2}
\left(\frac{\eta^{n_2l}}{\xi_2^l}\cdot\frac{e_i}{\eta^{m_3}
}\right)=q^{-l-1}\cdot\left(\sum_{j=k_{i,l,1}'
}^{k_{i,l,2}'}c_{i,l,j}'
\cdot\frac{\xi_1^j}{\eta^{jn_1}}\right)\end{equation} for all
$0\leq l\leq m$, where $k_{i,l,1}',k_{i,l,2}'$ are integers and
$c_{i,l,j}'\in R$. Therefore \begin{eqnarray}\lefteqn{
\psi^*\left( \frac{z\cdot
\partial}{\partial z}{\rm Trace}_{\varphi_2}
\left(\frac{\eta^{n_2l}}{\xi_2^l}\cdot\frac{e_i}{\eta^{m_3}
}\right)\right)=}\nonumber\\
&&\frac{\sum_{j=k_{i,l,1}'}^{k_{i,l,2}'}c_{i,l,j}'
\cdot\frac{\eta^{jn_1}}{\xi_1^j}}{\left( \sum_{j=0}^{n_2}
c_j\cdot\frac{\eta^{jn_1}}{\xi_1^j}\right)^{l+1}}
=\frac{\sum_{j=k_{i,l,1}'}^{k_{i,l,2}'}c_{i,l,j}'
\cdot\frac{\eta^{(j-n_2l-n_2)n_1}}{\xi_1^{j-n_2l-n_2}}}{\left(
\sum_{j=0}^{n_2} c_j\cdot\frac{\xi_1^{n_2-j}}{\eta^{(n_2-j)
n_1}}\right)^{l+1}}\label{da0530}\end{eqnarray} Since $c_{n_2}$ is
a unit in $R$, so the divisor determined by $\sum_{j=0}^{n_2}
c_j\cdot\frac{\xi_1^{n_2-j}}{\eta^{(n_2-j) n_1}}$ on $X$ does not
intersect with $D_1$. Moreover $E_P$ does not intersect with $D_1$
on $X$, so the residue
\begin{equation}\label{da0540}{\rm Res}_{D_1/Y}
\left(\omega\cdot \sum_{i\in
I_1}\frac{\xi_2^l}{\eta^{n_2l}}\cdot\frac{u_i}{\eta^{m_2}}\cdot
\varphi_2^*\left(\psi^*\left( \frac{z\cdot\partial}{\partial
z}{\rm Trace}_{\varphi_2}\left( \frac{\eta^{n_2l}}{\xi_2^l}\cdot
\frac{e_i}{\eta^{m_3}}\right)\right) \right)\right)
\end{equation} is an element in $R$ for all $0\leq l\leq m$.
Moreover,
\begin{eqnarray}\lefteqn{\frac{1}{2\pi\sqrt{-1}}\cdot
\int_{\Gamma_{\sigma}}
 \omega\cdot
\sum_{i\in I_1}\tau^{-l}\cdot\frac{u_i}{\eta^{m_2}}\cdot
\varphi_2^*\left(\psi^*\left(\frac{z\cdot\partial}{\partial z}
{\rm Trace}_{\varphi_2}\left( \tau^{l} \cdot
\frac{e_i}{\eta^{m_3}}\right)\right)
\right)}\nonumber\\
&=&{\rm Res}_{D_1/Y} \left(\omega\cdot \sum_{i\in
I_1}\frac{\xi_2^l}{\eta^{n_2l}}\cdot\frac{u_i}{\eta^{m_2}}\cdot
\varphi_2^*\left(\psi^*\left(\frac{z\cdot\partial}{\partial z}
{\rm Trace}_{\varphi_2}\left( \frac{\eta^{n_2l}}{\xi_2^l}\cdot
\frac{e_i}{\eta^{m_3}}\right)\right)
\right)\right)\nonumber\\
&&\label{da0560}\end{eqnarray} under $\sigma$. By (\ref{da0560})
and $\frac{(2m)!}{m!\cdot m!}\cdot b_{m,l}$ is an integer for all
$0\leq l\leq m$, so there exists an element in $R$ that is equal
to (\ref{da0502}) under the complex embedding $\sigma$.

Let $I_{2,\sigma}, U_{i,\sigma}$ be the elements defined in
condition (\trm{2}). Assume $B\cdot\left|\frac{\omega}{dz}\right|$
is greater than $\left|\frac{u_j}{\eta^{m_2}}\right|$,
$\left|\frac{e_j}{\eta^{m_3}}\right|$, $
\left|\frac{\partial}{\partial z}\frac{e_j}{\eta^{m_3}}\right|$
and $\left|\frac{\partial\tau}{\partial z}\right|$ on
$U_{i,\sigma}$ for all $i\in I_{2,\sigma}$ and $j\in I_1$, where
$B>1$. For $x\in U_{i,\sigma}$, let $G_x(s_{\omega}, \tau,z)$ be
the rational function on $\ca{P}_Y^1\otimes_\sigma\ca{C}$ defined
by
\begin{equation}\label{da0910} G_x(s_{\omega}, \tau,z)=\sum_{i\in
I_1}\sum_{l=0}^{m}b_{m,l}\cdot \tau(x)^{-l}\cdot
\frac{u_i}{\eta^{m_2}}(x)\cdot \frac{z\cdot\partial}{\partial
z}{\rm Trace}_{\varphi_2}\left( \tau^{l} \cdot
\frac{e_i}{\eta^{m_3}} \right)\end{equation} Then by
(\ref{da8940}) and condition (\trm{3}) satisfied by $\xi_2$, there
exists $\rho_2\in (0,\rho_1)$, and positive numbers $a_5,a_6$
determined by $\rho_2,B,n_1,X_{\sca{C}}$, such that
\begin{equation}\label{da0703} \left|G_x(s_{\omega}, \tau,z)(x')
-\frac{\omega}{dz}(x) \cdot {z(x')}\right|\leq \left(a_5m\cdot
\left|z(x')\right|^2 +n_1\cdot\rho_2^m\cdot
\left|z(x')\right|\right)\cdot\left|\frac{\omega}{dz}(x)\right|
\end{equation}
for all $m>a_6$ and $x'\in \ca{P}_Y^1\otimes_\sigma\ca{C}$
satisfying $\left|z(x')-z(x)\right|<\frac{3}{\|\xi_1\|}$, when
$\left|\frac{\xi_1}{\eta^{n_1}}(x)\right|<2$ and $\|\xi_1\|$ is
large enough.

 Assume
\begin{equation}\label{da0710}
\frac{\omega}{dz}(x_{i,\sigma})=\beta_{i,\sigma}\end{equation} for
$i\in I_{2,\sigma}$, where $\beta_{i,\sigma}\in\ca{C}$. For $i\in
I_{2,\sigma}$, let $\Gamma_{\sigma,i}$ be the path on
$U_{i,\sigma}$ determined by
$\left|\frac{\xi_1}{\eta^{n_1}}\right|=1$. For $x\in
\Gamma_{\sigma,i}$, by (\ref{da0703}), $\varphi_2^*(G_x)$ is
approximately equal to $\beta_{i,\sigma}\cdot z$ on
$\Gamma_{\sigma,i}$, when $\|\xi_1\|$ and $m$ are sufficiently
large. Since we  have
\begin{equation}\label{da0900}G_2(x)=\varphi_2^*(\psi^*(G_x))(x)
\end{equation} so $G_2$ is approximately
equal to $\beta_{i,\sigma}\cdot {\overline{z}}$ on
$\Gamma_{\sigma,i}$, i.e.\ $dG_2$ is approximately equal to
$\beta_{i,\sigma}\cdot \overline{\beta_{i,\sigma}^{-1}}\cdot
\overline{\omega}$ on $\Gamma_{\sigma,i}$, when $\|\xi_1\|$ and
$m$ are sufficiently large. So we will have
$\frac{1}{2\pi\sqrt{-1}}\cdot \int_{\Gamma_{\sigma,i}} \omega\cdot
G_2 $ is approximately equal to $\frac{\beta_{i,\sigma}^2}{
\|\xi_1\|^2}$, when $\|\xi_1\|$ and $m$ are sufficiently large.
Since there exists an element in $R$ that is equal to
(\ref{da0502}) under complex embedding $\sigma$, so we have
\begin{equation}\label{da0920}
\sum_{\sigma\in S}\log\left|2\sum_{i\in
I_{2,\sigma}}{\beta_{i,\sigma}^2}\right|> {[F:\ca{Q}]}\cdot
\left(\log \frac{m!\cdot m!}{(2m)!}+\log \|\xi_1\|^2\right)
\end{equation} when $\sum_{i\in
I_{2,\sigma}}{\beta_{i,\sigma}^2}$ is nonzero and $\|\xi_1\|,m$
are sufficiently large. From this we will have the following
theorem:

\begin{theorem}\label{da0810} Assume
\begin{equation}\label{da0812}
\left|\sum_{i\in
I_{2,\sigma}}{\beta_{i,\sigma}^2}\right|>a_9\cdot\|\omega\|^2
\end{equation} for all $\sigma\in S$, where $a_9$
is a positive number and $\|\omega\|$ denotes the canonical norm
[L1] of $\omega$. Then there exists $a_{7},a_8>1$ determined by
$a_9,B,n_1,X_{\sca{C}}$, such that
\begin{equation}\label{da0815}
\log\|\xi_1\|<a_7 +\frac{1}{2}\cdot\log\frac{2\sum_{\sigma\in
S}\left|\sum_{i\in
I_{2,\sigma}}{\beta_{i,\sigma}^2}\right|}{[F:\ca{Q}]}
\end{equation} when $\|\xi_1\|>a_8$.
\end{theorem}

The detail of the proof of {\bf Theorem \ref{da0810}} is written
in the last three pages of this paper. Note when $\|\omega\|$ is
large enough, $B$ can be chosen to be a constant determined by
$n_1,X_{\sca{C}}$. And when $\|\xi_1\|,\|\omega\|$ are large
enough, $\xi_1$ and $\omega$ can be chosen such that the
conditions required by {\bf Theorem \ref{da0810}} are satisfied.
Note
\begin{equation}\label{da0931} \frac{\sum_{\sigma\in
S}\left|\sum_{i\in
I_{2,\sigma}}{\beta_{i,\sigma}^2}\right|}{[F:\ca{Q}]}\leq
\|\omega\|^2\cdot O(1)\end{equation} And when $C_1=0$, we have
\begin{equation}\label{da0930}\log \|\xi_1\|=
\frac{n_1}{[F:\ca{Q}]}\omega_{X/Y}\cdot E_P+O(1)
\end{equation} So  (\ref{da0815}) will imply
an upper bound for $\omega_{X/Y}\cdot E_P$ determined by $n_1,
X_{\sca{C}}$, when $n_1$ is large enough under the conditions
(\trm{1})-(\trm{4}), where $C_1,C_2$ are assumed to be zero.

\section{The Proof}
Firstly we want to prove function $G_x$ constructed in the
introduction has the property (\ref{da8940}).

\begin{lemma}\label{da3098}
Let $V,v_0,v_1,\ca{P}_{\sca{C}}^1,M_m,f_5,
b_{m,l},n_1,\{\varrho_i:i\in
I\},x_i,U_{x_i}(r),\ca{P}_{\sca{C}}^1$ be the elements defined in
the introduction.  Let $m$ be a positive integer. For $x\in
\ca{P}_{\sca{C}}^1$, let $f_{4,x}\left(v_0,v_1,{m}\right)$ be the
rational function on $\ca{P}_{\sca{C}}^1$ defined by the
following:
\begin{equation}\label{da7021}
f_{4,x}\left(v_0, v_1,m\right)=\sum_{l=0}^{m}b_{m,l}\cdot
\frac{v_1^{-l}}{v_0^{-l}}(x)\cdot \frac{v_1^{l}}{v_0^{l}}
\end{equation}

Then there exists $r_1,\rho_1\in (0,1)$ that satisfies the
following:
\begin{enumerate}\setcounter{rom}{0}\item[\rmn]
For all $i,j\in I$ satisfying $i\neq j$, and for all $x\in
U_{x_i}(r_1)$ and $x'\in U_{x_j}(r_1)$, we have
\begin{equation}\label{da7022}
\left|f_{4,x}\left(v_0,v_1,{m}\right)(x')\right|<\rho_1^m
\end{equation} for all $m>0$.
\item[\rmn] For all $i\in I$ and $x',x\in U_{x_i}(r_1)$ and for
all $m>0$, we have
\begin{equation}\label{da7024}
\left|f_{4,x}\left(v_0,v_1,{m}\right)(x')-1\right|\leq 2m\cdot
\left|\frac{v_{1}}{v_0}(x')-\frac{v_{1}}{v_0}(x)\right|
\end{equation} when $\left|\frac{v_{1}}{v_0}(x')-\frac{v_{1}}{v_0}(x)
\right|<\frac{1}{3m}$.
\end{enumerate}
\end{lemma}

\begin{proof}
Let $x\in\ca{P}_{\sca{C}}^1$ be a closed point. Assume
\begin{equation}\label{da1000}\frac{v_1}{v_0}(x)=\varrho
\end{equation} where $\varrho\in \ca{C}$.
Let $v_{x,1}$ and $v_{x,0}$ be the elements in $V$, such that
\begin{equation}\label{da8342}v_{x,1}=\frac{1}{\left(1+\left|
\varrho\right|^2\right)^{\frac{1}{2}}}\cdot\left( v_1-\varrho\cdot
v_0\right)
\end{equation}
\begin{equation}\label{da8343}v_{x,0}=
\frac{1}{\left(1+\left|\varrho
\right|^2\right)^{\frac{1}{2}}}\cdot\left( \overline{\varrho}\cdot
v_1+v_0\right)
\end{equation} on $\ca{P}_{\sca{C}}^1$.
By (\ref{da8342}) (\ref{da8343}), we have
\begin{equation}v_0=\frac{1}{
\left(1+\left|\varrho\right|^2\right)^{\frac{1}{2}}}
\cdot\left(v_{x,0}-\overline{\varrho}\cdot v_{x,1}\right)
\label{da9680}\end{equation}
\begin{equation}\label{da9690}v_1=\frac{1}{
\left(1+\left|\varrho\right|^2\right)^{\frac{1}{2}}}
\cdot\left(\varrho\cdot v_{x,0}+v_{x,1}\right)
\end{equation}

Note \begin{equation}\label{da9691}
\left|\frac{v_{x,0}}{v_0}\right|^2+\left|\frac{v_{x,1}}{v_0}
\right|^2=\left|\frac{v_1}{v_0}\right|^2+1\end{equation} on
$\ca{P}_{\sca{C}}^1$. So  for $l=0,1$ and $x'\in U_{x_j}(r_1)$, we
have
\begin{equation}\label{da0060}\left|\frac{v_{x,l}}{v_0}(x')
\right|\leq\left(1+\left|\frac{v_1}{v_0}(x')
\right|^2\right)^{\frac{1}{2}}
\end{equation}

Note for $0\leq l\leq {m}$, we have
\begin{eqnarray}\lefteqn{ f_5\left(
\frac{(2m)!}{(2m-l)!\cdot l!}\cdot
v_{x,0}^{2m-l}v_{x,1}^{l}\right)=}\nonumber\\
&&\sum_{l_1=0}^{l} \frac{m!}{(m-l_1)!\cdot l_1!}\cdot
\frac{m!}{(m-l+l_1)!\cdot
(l-l_1)!}\cdot\nonumber\\
&&{v_{x,0}^{m-l+l_1} v_{x,1}^{l-l_1}}\otimes {v_{x,0}^{m-l_1}
v_{x,1}^{l_1}} \label{da9801}\end{eqnarray} And for ${m}<l\leq
2m$, we have
\begin{eqnarray}\lefteqn{ f_5\left(
\frac{(2m)!}{(2m-l)!\cdot l!}\cdot
v_{x,0}^{2m-l}v_{x,1}^{l}\right)=}\nonumber\\
&&\sum_{l_1=l}^{2m} \frac{m!}{(2m-l_1)!\cdot (l_1-m)!}\cdot
\frac{m!}{(l_1-l)!\cdot (m+l-l_1)!}
\cdot\nonumber\\
&& {v_{x,0}^{2m-l_1} v_{x,1}^{l_1-m}}\otimes{v_{x,0}^{l_1-l}
v_{x,1}^{m+l-l_1}}\label{da3340}\end{eqnarray}

Let $f_{3,x}$ be the linear map from $M_{2m}$ to the vector space
of rational functions on $\ca{P}_{\sca{C}}^1$ defined by $f_5$
followed by the map from $M_m\otimes M_m$ to rational functions on
$\ca{P}_{\sca{C}}^1$ that maps $t_1\otimes t_2$ to
$\frac{t_1}{v_0^m}(x)\cdot \frac{t_2}{v_0^m}$. Since
$v_{x,1}(x)=0$ and
$\frac{v_{x,0}}{v_0}(x)=\left(1+|\varrho|^2\right)^{\frac{1}{2}}$,
so by (\ref{da9801}) (\ref{da3340}), we have
\begin{eqnarray}
f_{3,x}\left( v_{x,0}^{2m-l}v_{x,1}^{l}\right)= \frac{(2m-l)!\cdot
l!}{(2m)!}\cdot\left(1+|\varrho|^2\right)^{\frac{m}{2}}\cdot
\frac{m!}{(m-l)!\cdot l!}\cdot
\frac{v_{x,0}^{m-l}v_{x,1}^l}{v_0^m}\label{da0100}\end{eqnarray}
when $0\leq l\leq m$, and $f_{3,x}\left(
v_{x,0}^{2m-l}v_{x,1}^{l}\right)$ is equal to $0$ when $l>m$. Note
\begin{equation}\label{da0101}
\frac{(2m-l)!\cdot l!}{(2m)!}\cdot\frac{m!}{(m-l)!\cdot l!} \leq
2^{-l}\end{equation} for $0\leq l\leq m$, and
\begin{equation}\label{da0103}
\frac{(2m-l)!\cdot l!}{(2m)!}\cdot\frac{m!}{(m-l)!\cdot l!} \leq
2^{-\frac{m}{2}}\cdot
3^{-l+\frac{m}{2}}=\frac{3^{\frac{m}{2}}}{2^{\frac{m}{2}}}\cdot
3^{-l}\end{equation} for $\left[\frac{m}{2}\right]+1\leq l\leq m$.

Let $r_1$ be a positive number. Assume $x$ is a point in
$U_{x_i}(r_1)$. Then we have
\begin{eqnarray}\lefteqn{
v_0^{{m}}v_1^{{m}}= \left(1+\left|\varrho\right|^2\right)^{-{m}}
\cdot\left(v_{x,0}-\overline{\varrho}\cdot
v_{x,1}\right)^m\cdot\left(\varrho\cdot v_{x,0}+v_{x,1}\right)^m
}\nonumber\\
&=&\left(1+\left|\varrho\right|^2\right)^{-{m}}
\cdot\left(\varrho\cdot
v_{x,0}^2+\left(1-\left|\varrho\right|^2\right)\cdot
v_{x,0}v_{x,1}-\overline{\varrho}\cdot v_{x,1}^2
\right)^m\label{da3300}\end{eqnarray} Let
\begin{equation}\label{da3310}\lambda_1=\left|\frac{v_{x,1}}{
v_{x,0}}\right|
\end{equation} Consider the subset determined by
 $\lambda_1\leq {2}$ over $U_{x_j}(r_1)$, where
$j\neq i$. By (\ref{da0100}) (\ref{da0101}), we have
\begin{eqnarray}\lefteqn{\left|f_{3,x}(v_{x,0}^{2m-l}v_{x,1}^{l}
)\right| \leq 2^{-l}\cdot \left(1+|\varrho|^2\right)^{\frac{m}{2}}
\cdot\left|\frac{v_{x,0}^{m-l}v_{x,1}^l}{v_0^m}\right|}\nonumber\\
&\leq& 2^{-l}\cdot \left(1+|\varrho|^2\right)^{\frac{m}{2}}
\cdot\left|\frac{v_{x,0}^{m}}{v_0^m}\right|\cdot\lambda_1^l
\hspace{1in}\label{da3320}\end{eqnarray} By (\ref{da9691}), we
have
\begin{equation}\label{da3330} \left|\frac{v_{x,0}}{v_0}\right|^2
\cdot \left(1+\lambda_1^2\right)=\left|\frac{v_1}{v_0}\right|^2+1
\end{equation}
By $\lambda_1\leq {2}$, we have
\begin{equation}\label{da3350}
\left(1+\frac{\lambda_1^2}{4}\right)^2<1+\lambda_1^2
\end{equation} By $j\neq i$, we have $\lambda_1>0$.
So when $r_1>0$ is small enough, we have
\begin{eqnarray}\lefteqn{\left|\frac{v_{x,0}^m}{v_0^m}\right|
\cdot
\left(\left|\varrho\right|+\frac{\left|1-\left|\varrho\right|^2
\right|\cdot\lambda_1}{2}+\frac{\left|\varrho\right|
\cdot\lambda^2_1}{4}\right)^m}\nonumber\\
&<&\left|\varrho\right|^m \cdot
\left(1+\lambda_1^2\right)^\frac{m}{2} \cdot
\left|\frac{v_{x,0}^m}{v_0^m}\right|\cdot \rho_0^m
\nonumber\\
&=&\left|\varrho\right|^m \cdot
\left(\left|\frac{v_1}{v_0}\right|^2+1\right)^{\frac{m }{2}}\cdot
\rho_0^m\label{da3360}\end{eqnarray} where $\rho_0\in (0,1)$. So
when we take $r_1$ small enough, such that
\begin{equation}\label{da3370}\left(1+(1-r_1)^2\right)^{-1}
\cdot \left(1+(1+r_1)^2\right)\cdot\rho_0^2<1\end{equation} By
(\ref{da3300}) (\ref{da3320}) (\ref{da3360}) (\ref{da3370}), we
have
\begin{eqnarray}\lefteqn{\left|f_{3,x}(v_0^{{m}}v_1^{{m}})\right|
\leq \left(1+\left|\varrho\right|^2\right)^{-\frac{m}{2}}
\cdot\left|\frac{v_{x,0}^m}{v_0^m}\right| }\nonumber\\
&&\cdot
\left(\left|\varrho\right|+\frac{\left|1-\left|\varrho\right|^2
\right|\cdot\lambda_1}{2}+\frac{\left|\varrho\right|
\cdot\lambda_1^2}{4}\right)^m\nonumber\\
&<&\left(1+\left|\varrho\right|^2\right)^{-\frac{m}{2}}
\cdot\left|\varrho\right|^m \cdot
\left(\left|\frac{v_1}{v_0}\right|^2+1\right)^{\frac{m }{2}}
\cdot \rho_0^m\nonumber\\
&<&\rho_6^m \cdot\left|\varrho\right|^{{m}}
\label{da0110}\end{eqnarray} over $U_{x_j}(r_1)$, where $j\neq i$
and $\rho_6\in (0,1)$.

Now consider the subset determined by
 $\lambda_1\geq {2}$ over $U_{x_j}(r_1)$, where
$j\neq i$. By (\ref{da9691}), we have
\begin{equation}\label{da3450}\left|
\frac{v_{x,1}}{v_0}\right|^2\cdot \left(1+\lambda_1^{-2}\right)
=1+\left|\frac{v_1}{v_0}\right|^2\end{equation} So we have
\begin{equation}\label{da3460}
\left| \frac{v_{x,1}}{v_0}\right|^2\geq \frac{4}{5} \cdot
\left(1+\left|\frac{v_1}{v_0}\right|^2\right)\end{equation}
Therefore over $U_{x_j}(r_1)$, by (\ref{da9691}), we have
\begin{equation}\label{da3470}
\left|\frac{v_{x,0}}{v_0}\right|^2\leq \frac{1}{5} \cdot
\left(1+\left|\frac{v_1}{v_0}\right|^2\right) < \frac{1}{5} \cdot
\left(1+(1+r_1)^2\right)\end{equation}

Assume \begin{equation}\label{da3490} v_0^mv_1^m =\sum_{i=0}^{2m}
b_{m,i,x}\cdot v_{x,0}^{2m-i}v_{x,1}^i\end{equation} where
$b_{m,i,x}\in\ca{C}$. By (\ref{da0100}) (\ref{da0103})
(\ref{da3470}) (\ref{da0060}), over $U_{x_j}(r_1)$ for
$\left[\frac{m}{2}\right]+1\leq l\leq 2m$, we have
\begin{eqnarray}\lefteqn{
\left|f_{3,x}\left( v_{x,0}^{2m-l}v_{x,1}^{l}\right)\right| <
}\nonumber\\
&&\frac{3^\frac{m}{2}}{2^\frac{m}{2}}\cdot 3^{-l}\cdot
(1+(1+r_1)^2)^{\frac{m}{2}} \cdot\frac{1}{5^\frac{m-l}{2}}\cdot
(1+(1+r_1)^2)^{\frac{m}
{2}}\nonumber\\
&=&\frac{3^\frac{m}{2}}{10^\frac{m}{2}}\cdot
(1+(1+r_1)^2)^{{m}}\cdot
\left(\frac{5}{9}\right)^\frac{l}{2}\label{da3480}\end{eqnarray}
By (\ref{da3490}) (\ref{da3480}) (\ref{da3300}), we have
\begin{eqnarray}\lefteqn{
\sum_{i=\left[\frac{m}{2}\right]+1}^{2m} \left|b_{m,i,x}\cdot
f_{3,x}\left(v_{x,0}^{2m-i}v_{x,1}^i
\right)\right| <}\nonumber\\
&&\frac{\left(1+(1+r_1)^2\right)^{{m}}
}{\left(1+\left|\varrho\right|^2\right)^{{m}}} \cdot
\frac{3^\frac{m}{2}}{10^\frac{m}{2}}\cdot
\left(|\varrho|+\frac{5^{\frac{1}{2}}\cdot
\left|1-|\varrho|^2\right|}{3}+ \frac{5\cdot
|\varrho|}{9}\right)^m \label{da3440}\end{eqnarray} When
$r_1\longrightarrow 0^+$, we have $|\varrho|\longrightarrow 1$. So
(\ref{da3440}) implies
\begin{equation}\label{da3500}
\sum_{i=\left[\frac{m}{2}\right]+1}^{2m}\left|b_{m,i,x}\cdot
f_{3,x}\left(v_{x,0}^{2m-i}v_{x,1}^i \right)\right|
<|\varrho|^m\cdot\rho_4^m
\end{equation}
where $\rho_4\in (0,1)$, when $r_1>0$ is small enough. By
(\ref{da3470}) (\ref{da0060}) (\ref{da0100}) (\ref{da0101}), over
$U_{x_j}(r_1)$ for $0\leq l\leq \left[\frac{m}{2}\right]$, we have
\begin{equation}\label{da3510}
\left|f_{3,x}\left( v_{x,0}^{2m-l}v_{x,1}^{l}\right)\right| \leq
\frac{1}{5^\frac{m}{4}}\cdot 2^{-l}\cdot (1+(1+r_1)^2)^{{m}}
\end{equation} By (\ref{da3510}) (\ref{da3300}), we have
\begin{eqnarray}\lefteqn{\sum_{i=0}^{\left[\frac{m}{2}\right]}
\left|b_{m,i,x}\cdot f_{4,x}\left(v_{x,0}^{2m-i}v_{x,1}^i
\right)\right| <}\nonumber\\
&&\frac{\left(1+(1+r_1)^2\right)^{{m}}
}{\left(1+\left|\varrho\right|^2\right)^{{m}}} \cdot
\frac{1}{5^\frac{m}{4}}\cdot \left(|\varrho|+
\frac{\left|1-|\varrho|^2\right|}{2}+ \frac{|\varrho|}{4}\right)^m
\label{da3520}\end{eqnarray} When $r_1\longrightarrow 0^+$, we
have $|\varrho|\longrightarrow 1$. So (\ref{da3520}) implies
\begin{equation}\label{da3530}
\sum_{i=0}^{\left[\frac{m}{2}\right]}\left|b_{m,i,x}\cdot
f_{3,x}\left(v_{x,0}^{2m-i}v_{x,1}^i \right)\right|
<|\varrho|^m\cdot\rho_5^m
\end{equation}
where $\rho_5\in (0,1)$, when $r_1>0$ is small enough. Then
(\ref{da0110}) (\ref{da3500}) (\ref{da3530}) imply (\trm{1}) is
true.

Over $U_{x_i}(r_1)$, by (\ref{da3300}), we have
\begin{equation}
f_{3,x}(v_0^{{m}}v_1^{{m}})=\frac{\varrho^m\cdot
f_{3,x}(v_{x,0}^{2m})}{ \left(1+\left|\varrho\right|^2\right)^{m}}
+\sum_{i=1}^{2m}b_{i,2}'\cdot f_{3,x}(v_{x,0}^{2m-i}v_{x,1}^i)=
{\varrho^{{m}}}\cdot \left(1 +u
\right)\label{da0120}\end{equation} where $b_{i,2}'\in\ca{C}$ and
$u$ is an analytic function on $U_{x_i}(r_1)$. By (\ref{da3300})
and (\ref{da0120}), we can see  $u$ satisfies
$\left|u(x')\right|\leq 2m\cdot
\left|\frac{v_{1}}{v_0}(x')-\frac{v_1}{v_0}(x)\right|$ for all
$x',x\in U_{x_i}({r_1})$ satisfying $
\left|\frac{v_{1}}{v_0}(x')-\frac{v_1}{v_0}(x)\right|<\frac{1}{3m}
$ when $r_1$ is small enough.  So (\trm{2}) is true.
\end{proof}

\begin{theorem}\label{da9000} Let $\{x_i:i\in I\},
U_{x_i}(r),r_1,\rho_1,b_{m,l},v_0,v_1$ be the elements in {\bf
Lemma \ref{da3098}}. Let $r_0>0$ be a constant. Let $z$ be an
analytic function on $\bigcup_{i\in I}U_{x_i}(r_1)$ in
$\ca{P}_{\sca{C}}^1$ that satisfies the following:
\begin{enumerate}\setcounter{rom}{0}
\item $dz(x)\neq 0$ for all $x\in \bigcup_{i\in I} U_{x_i}(r_1)$
satisfying $\left|z(x)\right|<r_0$. \item For all $i\in I$ and
$\alpha\in\ca{C}$ satisfying $|\alpha|<r_0$, there exists a unique
point $x\in U_{x_i}(r_1)$, such that $z(x)=\alpha$.
\end{enumerate}

Let $\{u_i',e_i':i\in I_1\}$ be a set of analytic functions on $
\bigcup_{i\in I}U_{x_i}(r_1)$ satisfying
\begin{equation}\label{da9010}\sum_{i\in I_1}u_i'\cdot e_i'=0
\end{equation} Let \begin{equation}\label{da9020}
\omega'=\sum_{i\in I_1}u_i'\cdot de_i'\end{equation} Let
$U(r_0)\subset \bigcup_{i\in I}U_{x_i}(r_1)$ be the subset
determined by $ |z|<r_0$. Assume there exists $B'>1$, such that
$B'\cdot\left|\frac{\omega'}{dz}\right|$ is greater than
$\left|u_i'\right|$, $\left|e_i'\right|$, $\left|\frac{\partial }{
\partial z}\frac{v_1}{v_0}\right|$ and
$\left|\frac{\partial e_i'}{
\partial z}\right|$ for all $i\in I_1$ on $U(r_0)$.
Let $\varphi_2$ be the map from $U(r_0)$ to the complex plane
defined by $z$. Let $U_0$ be the image of $U(r_0)$ in the complex
plane. For $x\in U(r_0)$, let $G_x$ be the analytic function on
$U_0$ defined by
\begin{equation}\label{da9040}
G_x=\sum_{j\in I_1}\sum_{l=0}^{m}b_{m,l}\cdot
\frac{v_1^{-l}}{v_0^{-l}}(x)\cdot u_j'(x)\cdot \frac{z\cdot
\partial}{\partial z}{\rm Trace}_{\varphi_2}\left(\frac{
v_1^{l}}{v_0^{l}}\cdot e_j'\right)\end{equation} Let $\rho_2$ be a
number in $(0,\rho_1)$. Then there exists $a_5,a_6>0$ determined
by $\rho_2,n_1,B',\omega'$, such that
\begin{equation}\label{da9060} \left|G_x(x')-
\frac{\omega'}{dz}(x)\cdot {z(x')}\right|\leq \left(a_5m\cdot
\left|z(x')\right|^2 +n_1\cdot\rho_2^m\cdot
\left|z(x')\right|\right)\cdot\left|\frac{\omega'}{dz}(x)\right|
\end{equation}
for all $m>a_6$ and all $x'\in U_0$ satisfying
$\left|z(x')-z(x)\right|<\frac{1}{3m}$.
\end{theorem}

\begin{proof} Let $V'$ be the vector space
generated by $e_i'$ over $\ca{C}$, where $i\in I_1$. Let $h'$ be
the Hermitian metric on $V'$, defined by
\begin{equation}\label{da9061}
\left<e_i',e_j'\right>_{h'}=\frac{\sqrt{-1}}{2}\cdot
\int_{U(r_0)}e_i'\cdot\overline{e_j'}\cdot dz\wedge
d\overline{z}\end{equation} For $x\in U(r_0)$, let
$\{e_{x,0},e_{x,1},\cdots,e_{x,n_3} \}$ be a set of basis of $V'$
over $\ca{C}$, that satisfies the following:
\begin{enumerate}\item $e_{x,0}$ is orthogonal to the subspace
of $V'$ generated by the elements that vanish at point $x$ under
 the Hermitian metric $h'$.
\item ${e_{x,1}}(x)=0$, and $d{e_{x,1}}-dz$ vanishes at point $x$,
and $e_{x,1}$ is orthogonal to the subspace of $V'$ generated by
the elements that vanish at point $x$ with order $\geq 2$ under
 the Hermitian metric $h'$.
\item $e_{x,i}$ vanishes at point $x$ with order $\geq 2$ for all
$2\leq i\leq n_3$.
\end{enumerate}
Let $\{u_{x,0},\cdots,u_{x,n_3}\}$ be the set of analytic
functions on $U(r_0)$, such that
\begin{equation}\label{da8200}\sum_{i\in I_1}u_i'\otimes e_i'
=\sum_{i=0}^{n_3}u_{x,i}\otimes e_{x,i}
\end{equation}

Let \begin{equation} \label{da8201} G_{1,x}= \sum_{j\in
I_1}\sum_{l=0}^{m}b_{m,l}\cdot \frac{v_1^{-l}}{v_0^{-l}}(x)\cdot
u_j'(x)\cdot\left(\frac{ v_1^{l}}{v_0^{l}}\cdot
e_j'\right)\end{equation} Since
\begin{equation}\label{da8210}\sum_{i=0}^{n_3}u_{x,i}\cdot
e_{x,i}=0
\end{equation} and $e_{x,i}(x)=0$ for $i>0$, so
\begin{equation}\label{da8220} u_{x,0}(x)=0
\end{equation} By (\ref{da8220}), we have
\begin{eqnarray}\lefteqn{
G_{1,x}= \sum_{i=1}^{n_3}\sum_{l=0}^{m}b_{m,l}\cdot
\frac{v_1^{-l}}{v_0^{-l}}(x)\cdot {u_{x,i}}(x) \cdot
\frac{v_1^{l}}{v_0^{l}} \cdot {e_{x,i}}
}\nonumber\\
&=& f_{4,x}\left({v_0},v_1,m\right)\cdot \sum_{i=1}^{n_3}
{u_{x,i}}(x)\cdot {e_{x,i}}
\hspace{.7in}\label{da8132}\end{eqnarray}

Since
\begin{equation}\label{da8230} \sum_{i=0}^{n_3}
{u_{x,i}}\cdot d{e_{x,i}}=\omega'\end{equation} and $e_{x,i}$
vanishes at $x$ with order $\geq 2$ for $i>1$, so we have
\begin{equation}\label{da8250}{u_{x,1}}
(x)=\frac{\omega'}{dz}(x)
\end{equation}

Since $B'\cdot\left|\frac{\omega'}{dz}\right|$ is greater than
$|u_i'|,|e_i'|,\left|\frac{\partial}{\partial z}\frac{v_1}{v_0}
\right|, \left|\frac{\partial}{\partial z}e_i'\right|$ on
$U(r_0)$, so by {\bf Lemma~\ref{da3098}}, we have
\begin{equation}\label{da8251}
\left|\frac{z\cdot\partial}{\partial z}G_{1,x}\right|< \rho_2^m
\cdot\left|\frac{\omega'}{dz}(x)\right|\cdot \left|z\right|
\end{equation} on $U(r_0)\cap U_{x_j}(r_1)$
when $x$ is a point in $U(r_0)\smallsetminus U_{x_j}(r_1)$, where
$\rho_2\in (0,\rho_1)$ and $m$ is greater than a number determined
by $n_1,\omega',\rho_2,B'$. By {\bf Lemma~\ref{da3098}}, and
${e_{x,i}}$ vanishes at $x$ with order $\geq 2$ for $i>1$, and
${e_{x,1}}$ vanishes at $x$, and
\begin{equation}\label{da8252}\frac{\partial
{e_{x,1}}}{\partial z} (x)=1
\end{equation} we have
\begin{equation}\label{da8254}
\left|\frac{z\cdot\partial}{\partial z}G_{1,x}(x')- {z(x')}\cdot
\frac{\omega'}{dz}(x)\right|\leq a_5'm\cdot
\left|z(x')\right|^2\cdot\left|\frac{\omega'}{dz}(x)\right|
\end{equation}
when $x,x'\in U(r_0)\cap U_{x_j}(r_1)$ satisfying
$\left|z(x)-z(x')\right| <\frac{1}{3m}$, where $a_5'$ is a
positive number determined by $n_1,\omega',B',\rho_2$. Therefore
we have
\begin{equation}\label{da8253}
\left|G_x(x')- {\frac{\omega'}{dz}(x)\cdot z(x')}\right|<
\left(a_5m\cdot \left|z(x')\right|^2+n_1\cdot\rho_2^m\cdot
\left|z(x')\right| \right)\cdot\left|\frac{\omega'}{dz}(x)\right|
\end{equation} for all $x'\in U_0$ satisfying $\left|z(x)-z(x')\right|
<\frac{1}{3m}$, where
 $a_5,a_6>0$ are determined by $n_1,\omega',\rho_2,B'$. So our
Lemma is true.
\end{proof}

\hspace{1in}

Now we want to prove the function  $G_2$ on arithmetic surface $X$
constructed in the introduction have the properties that imply
{\bf Theorem \ref{da0810}}.

\begin{lemma}\label{da0500} For complex embedding
$\sigma:F\hookrightarrow\ca{C}$, let $\Gamma_\sigma$ be the path
determined by $\left|\frac{\xi_1}{\eta^{n_1}}\right|=1$ on
$X_\sigma$. Then there exists an element in $R$ that is equal to
\begin{eqnarray}\frac{(2m)!}{m!\cdot m!}\cdot
\frac{1}{2\pi\sqrt{-1}}\cdot \int_{\Gamma_{\sigma}} \omega\cdot
G_2 \label{da0501}\end{eqnarray} over complex embedding $\sigma$
for all $\sigma\in S$.
\end{lemma}

\begin{proof} This Lemma is proved in the introduction.
\end{proof}

\begin{lemma}\label{da0700}
 Assume
\begin{equation}\label{da0711}
\frac{\omega}{dz}(x_{j,\sigma})=\beta_{j,\sigma}\end{equation} for
$j\in I_{2,\sigma}$, where $\beta_{j,\sigma}\in\ca{C}$. Assume
$dz$ is not equal to zero at any point in the open set $U(r_0)$,
where $U(r_0)$ is the subset of $X_\sigma$ determined by
$\left|z\right|<r_0$, and $r_0$ is a positive number. Assume
$B\cdot\left|\frac{\omega}{dz}\right|$ is greater than
$\left|u_j\right|$, $\left|e_j\right|$, $
\left|\frac{\partial}{\partial z}e_j\right|$ and
$\left|\frac{\partial\tau}{\partial z}\right|$ on $U(r_0)$ for all
$j\in I_1$, where $B>1$.

Then there exists $\rho_2\in (0,\rho_1)$, and $a_{10},a_{11}>1$
determined by $\rho_2,B,n_1,X_{\sca{C}}$, such that
\begin{equation}\label{da0701} \left|\frac{1}{2\pi\sqrt{-1}}\cdot
\int_{\Gamma_{\sigma}} \omega\cdot G_2-\sum_{j\in
I_{2,\sigma}}\frac{\beta_{j,\sigma}^2}{
\|\xi_1\|^2}\right|<\left(a_{10}m\cdot\|\xi_1\|^{-3}+
\rho_2^m\cdot\|\xi_1\|^{-2}\right)\cdot \|\omega\|^2\end{equation}
for all $m>a_{11}$, when $\|\xi_1\|>a_{10}m$.
\end{lemma}

\begin{proof} Assume $\|\xi_1\|$ is sufficiently large.
For $j\in I_{2,\sigma}$,
let $U_{j,\sigma}$ be the open set
defined in property (\trm{2}) in the introduction. Let
$\Gamma_{\sigma,j}$ be the path on $U_{j,\sigma}$ determined by
$\left|\frac{\xi_1}{\eta^{n_1}}\right|=1$. Let $x$ be a point on
$\Gamma_{\sigma,j}$.
 Then there exists
$\theta\in [0,2\pi)$, such that
\begin{equation}\label{da0730}\frac{\xi_1}{\eta^{n_1}}(x)=
e^{i\theta}
\end{equation} And we have
\begin{equation}\label{da0760}\left|\tau(x)-\varrho_{i_j}\right|
<a_{12}\cdot \|\xi_1\|^{-1}\end{equation} where $a_{12}$ is a
positive number determined by $B,n_1,X_{\sca{C}}$.

Let \begin{equation}\label{da0770} G_x(s_{\omega},
\tau,z)=\sum_{l=0}^m\sum_{i\in I_1}b_{m,l}\cdot\tau^{-l}(x)\cdot
\frac{u_i}{\eta^{m_2}}(x)\cdot \frac{z\cdot\partial}{\partial
z}{\rm Trace}_{\varphi_2}\left( \tau^l\cdot
\frac{e_i}{\eta^{m_3}}\right)\end{equation} By (\ref{da0760}) and
{\bf Theorem~\ref{da9000}} and $\|\xi_1\|$ is large enough, we
have \begin{eqnarray}\lefteqn{
G_2(x)=\psi^*(G_x)(x)=}\nonumber\\
&&\psi^*\left({\frac{\omega}{dz}(x)\cdot z}
\right)(x)+O\left(\|\xi_1\|^{-2}\cdot\|\omega\|
\right)+O\left(\rho_2^m\cdot\left|\frac{\omega}{dz}(x)\right|\cdot
\|\xi_1\|^{-1}\right)\hspace{.3in}\label{da0720}\end{eqnarray}
where $\rho_2\in (0,\rho_1)$. Note
\begin{equation}\psi^*\left({\frac{\omega}{dz}(x)\cdot z}\right)(x)=
\frac{\frac{\omega}{dz}(x)\cdot z(x)^{-1}}{\|\xi_1\|^2}=
\frac{\beta_{j,\sigma}\cdot e^{-i\theta}}{\|\xi_1\|}\cdot
(1+O(\|\xi_1\|^{-1}))\label{da0740}\end{equation}

Therefore we have
\begin{eqnarray}\lefteqn{\frac{1}{2\pi\sqrt{-1}}\cdot
\int_{\Gamma_{\sigma,j}}
 \omega\cdot G_2=}\nonumber\\
&&\frac{1}{2\pi\sqrt{-1}}\cdot \int_{\Gamma_{\sigma,j}}
 \beta_{j,\sigma}\cdot dz\cdot \left(
 \frac{\beta_{j,\sigma}\cdot e^{-i\theta}}{\|\xi_1\|}+O(\|\xi_1\|^{-2}
\cdot\|\omega\|) +
O(\rho_2^m\cdot\left|\beta_{j,\sigma}\right|\cdot
\|\xi_1\|^{-1})\right)\nonumber\\
&=&\frac{1}{2\pi\sqrt{-1}}\cdot\int_0^{2\pi}
\beta_{j,\sigma}\cdot\frac{de^{i\theta}}{\|\xi_1\|} \cdot \left(
\frac{\beta_{j,\sigma}\cdot
e^{-i\theta}}{\|\xi_1\|}+O(\|\xi_1\|^{-2}\cdot\|\omega\|)+
O(\rho_2^m\cdot\left|\beta_{j,\sigma}\right|\cdot
\|\xi_1\|^{-1})\right)\nonumber\\
&=&\frac{\beta_{j,\sigma}^2}{\|\xi_1\|^2}+O\left(\|\xi_1\|^{-3}
\cdot\|\omega\|^2 \right)+O(\rho_2^m\cdot
\left|\beta_{j,\sigma}\right|^2 \cdot\|\xi_1\|^{-2})\label{da0750}
\end{eqnarray}
By (\ref{da0750}) and $\|\omega\|\geq O(|\beta_{j,\sigma}|)$, we
see our Lemma is true.
\end{proof}

\hspace{1in}

\begin{pot}{\ref{da0810}}
 By (\ref{da0812}), there exists $a_8>0$, such that
\begin{equation}\label{da0840}
\left|\sum_{j\in I_{2,\sigma}}\frac{\beta_{j,\sigma}^2}{
\|\xi_1\|^2}\right|>2\cdot \left(a_{10}m\cdot\|\xi_1\|^{-3}+
\rho_2^m\cdot\|\xi_1\|^{-2}\right)\cdot \|\omega\|^2\end{equation}
where $\|\xi_1\|>a_8$. By (\ref{da0840}) and {\bf Lemma
\ref{da0700}}, we have
\begin{equation}\label{da0850}
2\left|\sum_{j\in I_{2,\sigma}}\frac{\beta_{j,\sigma}^2}{
\|\xi_1\|^2}\right|>\left|\frac{1}{2\pi\sqrt{-1}}\cdot
\int_{\Gamma_{\sigma}} \omega\cdot
G_2\right|>\frac{1}{2}\left|\sum_{j\in
I_{2,\sigma}}\frac{\beta_{j,\sigma}^2}{
\|\xi_1\|^2}\right|\end{equation} Then by {\bf Lemma \ref{da0500}}
and the right hand side of (\ref{da0850}), we have
\begin{equation}\label{da0801}
\sum_{\sigma\in S}\log\left|\frac{1}{2\pi\sqrt{-1}}\cdot
\int_{\Gamma_{\sigma}} \omega\cdot G_2\right| \geq [F:\ca{Q}]\cdot
\log \left(\frac{m!\cdot m!}{(2m)!}\right)
\end{equation}

By (\ref{da0801}) and the left hand side of (\ref{da0850}), we
have
\begin{equation}\label{da0860}
\frac{2\sum_{\sigma\in S}\left|\sum_{i\in
I_{2,\sigma}}{\beta_{i,\sigma}^2}\right|}{
[F:\ca{Q}]\cdot\|\xi_1\|^2}
>\frac{m!\cdot m!}{(2m)!}\end{equation} Note
\begin{equation}\label{da0870}
\sum_{j=0}^{2m}\frac{(2m)!}{j!\cdot (2m-j)!}=2^{2m}\end{equation}
So we have \begin{equation}\label{da0880}\frac{(2m)!}{m!\cdot m!}
<{2^{2m}}
\end{equation}
Let $m$ be the minimal positive integer greater than $a_6$. Then
$m$ is  determined by $B,n_1,X_{\sca{C}}$. So by (\ref{da0860})
(\ref{da0880}), there exists $a_7,a_8$ determined by
$a_9,B,n_1,X_{\sca{C}}$ such that (\ref{da0815}) is true.
\end{pot}

\end{document}